%% file: galton-revisite-2025_12.tex
\documentclass[svgnames]{amsart}

\usepackage[utf8]{inputenc}
\usepackage{amsmath,amssymb,amsthm}
\usepackage{graphicx}
\usepackage[english]{babel}
\DeclareMathAlphabet{\mathbbo}{U}{bbold}{m}{n}
\usepackage{verbatim}
\newcommand{\ber}{\text{Ber}}
\newcommand{\floor}[1]{\lfloor #1 \rfloor}
\newcommand{\gestoch}{\succeq}
\newcommand{\lestoch}{\preceq}
\newcommand{\N}{\ensuremath{\mathbb{N}}}

\newcommand{\R}{\ensuremath{\mathbb{R}}}

\newcommand{\E}{\ensuremath{\mathbb{E}}}

\renewcommand{\P}{\ensuremath{\mathbb{P}}}

\renewcommand{\epsilon}{\varepsilon}

\renewcommand{\liminf}{\underline{\lim}\quad}

\newcommand{\miniop}[3]{%
\renewcommand{\arraystretch}{0.6}
\begin{array}{c}
{\scriptstyle #1}\\
#2\\
{\scriptstyle #3}
\end{array}
\renewcommand{\arraystretch}{1}}

\newcommand{\1}{\mathbbo{1}}
\newcommand{\Var}{\text{Var }}

\newtheorem{theo}{Theorem}
\newtheorem{lemme}{Lemma}
\newtheorem{coro}{Corollary}

\newtheorem{definition}{Definition}
\usepackage[most]{tcolorbox}
\tcbuselibrary{listingsutf8}
\input{lst-environnements}

\author{Olivier \textsc{Garet}}
\title[Probabilistic proof for non--survival at criticality]{Probabilistic proof for non--survival at criticality: the Galton--Watson process and more}
\begin{document}
\subjclass[2000]{60K35, 82B43.}
\keywords{Galton--Watson process, growth model,renormalization}

\maketitle
\begin{abstract}
In a famous paper, Bezuidenhout and Grimmett demonstrated that the contact process dies out at the critical point.Their proof technique has often been used to study the growth of population patterns.
The present text is intended as an introduction to their ideas, with examples of minimal technicality. In particular, we recover the basic theorem about 
Galton--Watson chains: except in a degenerate case, survival is possible only if the fertility rate exceeds $1$. The classical proof that is taught in classrooms is essentially analytic, based on generating functions and convexity arguments. Following the Bezuidenhout--Grimmett way, we propose a proof that is more consistent with probabilistic intuition.
We also study the survival problem for a cooperative model, mixing sexual and asexual reproduction.
\end{abstract}

\section{Introduction}
Inspired by an article by Grimmett and Marstrand on supercritical  percolation in dimension $d\ge 3$, Bezuidenhout and Grimmett have shown in a famous article that the contact process vanishes at the critical point.
Their proof technique has often been used to study various growth  models.

The implementation of their proof technique is usually quite technical, as it relies on a renormalization procedure with quite complicated events as a basic brick.

The purpose of this article is therefore to introduce this technique with growth models for which the implementation is much simpler.

Among the growth models, the most famous is the Galton--Watson process.
The basic theorem concerns the probability of survival as a function of fertility: except in degenerate cases, survival is possible only if the fertility rate exceeds $1$.
The proof that is usually taught  -- see for example  Benaïm--El Karoui~\cite{BEK} or Durrett~\cite{Durrett} --  is essentially analytic. It relies on generating functions and convexity arguments, which may seem rather frustrating or at least quite miraculous.

We propose here, inspired by the work of Bezuidenhout and Grimmett, to give a proof that is more in line  with the probabilistic intuition.

This gives an introduction to the ideas of Bezuidenhout and Grimmett, with a model that is probably the simplest of the models that can be considered.
We then continue with the study of the survival problem on an original model, mixing sexual and asexual reproduction.


In order to keep our text self-contained (maybe event suitable for a presentation to graduate students), the first section is devoted to the introduction of the Galton--Watson process with all the necessary results.
The new proof of the classical result comes in Section~2.
Section~3 is devoted to the introduction and the study of a new cooperative model, mixing sexual and asexual reproduction.
\section{Galton--Watson processes: definition and first properties}

Let $\nu,\mu$ be two distributions on $\N$.
The distribution $\nu$ is denoted as the offspring distribution, whereas $\mu$ 
is the distribution of the size of the initial population. 

We denote as  the 
Galton--Watson process with initial distribution $\mu$ and offspring distribution $\nu$
the Markov chain that starts with  $\mu$ as initial distribution, and whose transition matrix is given by
\begin{equation*}
p_{i,j}=
\begin{cases}
\nu^{*i}(j)\text{ if }i\ne 0\\
\delta_0(j)\text{ if }i=0
\end{cases}
\end{equation*}

One can build such a chain as follows:
Let $(X_i^n)_{i,j\ge 1}$, $Y_0$ be independent random variables with $Y_0\sim\mu$ and $X_i^n\sim\nu$ for every $i,n$.
Then, the sequence $(Y_n)_{n\ge 1}$ is recursively defined by

$$\forall n\ge 0\quad Y_{n+1}=\sum_{1\le i\le Y_n}X_i^n.$$
Then, $(Y_n)_{n\ge 0}$, is a Galton--Watson process with initial distribution $\mu$ and offspring distribution $\nu$.
The mean number of offspring  $m=\int_{\N} x\ d\nu(x)$ is denoted as the fertility.
If we define  $\mathcal{F}_n=\sigma(X_i^k,i\ge 1,k\le n)$, we have 
\begin{equation}
\label{puissance}
\E[Y_{n+1}|\mathcal{F}_n]=mY_n,\quad \E[Y_{n+1}]=m\E[Y_n]\text{ and }\E[Y_n]=m^n\E[Y_0]
\end{equation}
We define the time to extinction $\tau$ as follows: $\tau=\inf\{n\ge 0; Y_n=0\}$.

\begin{theo}
\label{mort}
If $m<1$, $\P(\tau>n)=O(m^n)$. Particularly, $\P(\tau<+\infty)=1$.
\end{theo}
\begin{proof}
With~\eqref{puissance}, we have $\P(\tau>n)\le\P(Y_n\ge 1)\le \E[Y_n]=m^n\E[Y_0]$.
\end{proof}

\begin{theo}
\label{galtonindep}
Let $(X_n)_{n\ge 0}$ and  $(Y_n)_{n\ge 0}$ be independent  Galton--Watson processes with the same offspring distribution  $\nu$. Then, 
 $(X_n+Y_n)_{n\ge 0}$ is also a  Galton--Watson process with $\nu$ as offspring distribution.
\end{theo}
\begin{proof}
Since $(X_n)_{n\ge 0}$ and  $(Y_n)_{n\ge 0}$ are independent Markov chains, $((X_n,Y_n))_{n\ge 0}$ is a Markov chain, with the transition matrix
$$p_{(x,a),(y,b)}=\nu^{* x}(a)\nu^{* y}(b).$$ Let us denote by 
$\P^{(x,y)}$ the distributions of the canonically associated Markov chains.
We must prove that if the function $f$ is defined by  $f(x,y)=x+y$, then $(f(X_n,Y_n))_{n\ge 0}$ is still a Markov Chain. To this aim, we apply the Dynkin criterion: it is sufficient to prove that whenever $x+y=r$, then
$\P^{(x,y)}(f(X_1,Y_1)=\ell)$ only depends on $r$ and $\ell$.
Also, under  $\P^{(x,y)}$, $X_1$ and $Y_1$ are independent random variables with  $\nu^{* x}$ and $\nu^{* y}$ as their respective distributions, so the distribution of
 $f(X_1,Y_1)$ is  $\nu^{* x}* \nu^{* y}=\nu^{* (x+y)}=\nu^{* r}$.
Finally, $\P^{(x,y)}(f(X_1,Y_1)=\ell)=\nu^{* r}(\{\ell\})$ and 
$(X_n+Y_n)_{n\ge 0}$ is a Galton--Watson process with $\nu$ as offspring distribution. Since the initial distribution is $\P_{X_0+Y_0}=\P_{X_0}*\P_{Y_0}=\mu_1*\mu_2$, we get the desired result.
\end{proof}

In the sequel, $\P^i$ denotes a probability measure under which
$(Y_n)_{n\ge 0}$ is a Galton--Watson process with initial distribution $\delta_{i}$ and offspring distribution $\nu$.

\begin{coro}
We have
\begin{itemize}
\item For each $n\ge 0$, $\P^n(\tau<+\infty)=\P^{1}(\tau<+\infty)^n$
\item For $n,\ell\ge 0$, $\P^n(\tau<+\infty|\mathcal{F}_{\ell})=\P^{1}(\tau<+\infty)^{Y_{\ell}}$.
\item For $n,\ell\ge 1$, we have $\P^n(\tau=+\infty)>0 \iff \P^{\ell}(\tau=+\infty)>0.$
\end{itemize}
\end{coro}

\begin{proof}
  Thanks to Theorem~\ref{galtonindep}, we have $$\P^{n+1}(\tau<+\infty)=\P^{n}(\tau<+\infty)\P^{1}(\tau<+\infty),$$ then  $\P^{n}(\tau<+\infty)=\P^{1}(\tau<+\infty)^n$ follows by natural induction. This gives the first item.
Then, the second item follows from the Markov property. The last point is obvious.
\end{proof}

\begin{coro}
\label{souschaine}
Let $T\ge 1$.  $(Y_{Tn})_{n\ge 0}$ is a Galton--Watson process with offspring distribution $\P^1_{Y_T}$.
\end{coro}
\begin{proof}
Since $(Y_n)$ is a Markov chain, it is well known that so does $(Y_{Tn})_{n\ge 0}$. Let us compute the transition probabilities.

Let $k\ge 1$.
Applying Theorem~\ref{galtonindep} ($k-1$ times), we see that if the processes $(Y^1_t)_{t\ge 0}$, $(Y^2_t)_{t\ge 0}, $\dots$ (Y^k_t)_{t\ge 0}$ are independent Galton--Watson processes with $\delta_1$ as their common initial distribution and $\nu$ as offspring distribution,
then $(Y^1_t+\dots Y^k_t)_{t\ge 0}$ is a  Galton--Watson process with  $\delta_{k}$ as initial distribution and $\nu$ as  offspring distribution.
Then,
$$\P^k(Y_T=\ell)=\P(Y^1_T+\dots Y^k_T=\ell)=\P_{Y^1_T}^{*k}(\ell).$$
Also, $\P^0(Y_T=\ell)=\delta_0(\ell)$: this gives the desired result.
\end{proof}

\section{A probabilistic proof}

In the first step of the proof, we show that a certain growth process may survive, with the idea that the process that we finally want to study will be compared to the surviving reference process.
In the present paper, the reference process is a Galton--Watson process too.
However in general, the reference process may belong to a related family.
For example, Bezuidenhout and Grimmett compared the contact process to a supercritical oriented percolation process.

\subsection{Survival in the supercritical phase}

\begin{theo}
\label{surcritique}
If $m>1$, then $\P^{1}(\tau=+\infty)>0$.
\end{theo}
\begin{proof}
Let $a$ with $1<a<m$. We have
$$\miniop{}{\lim}{M\to +\infty} \int x\wedge M \ d\nu= \int  x \ d\nu=m,$$ so there exists $M$ with  $\int x\wedge M \ d\nu>a$.  For $k\ge n$, we have

\begin{align*}
\P^k(Y_1<na)&=\P(X_1+\dots X_k<na)\\\le &\P(X_1\wedge M+\dots X_n\wedge M<na)\\& =\P(n\E[X_1\wedge M]-(X_1\wedge M+\dots X_n\wedge M))> (\E[X_1\wedge M]-a)n)\\& \le\frac{\Var X_1\wedge M}{(\E[X_1\wedge M]-a)n},
\end{align*}
by the  Tchebitchef inequality. 
Let us define $\phi(k,x)=\P^k(Y_1<x)$ and consider $n>c=\frac{\Var (X_1\wedge M)}{\E[X_1\wedge M]-a}$.\\
Let $t\ge 0$. By the Markov property, for each
 $A\in \mathcal{F}_t$ with $A\subset\{Y_t\ge n\}$, we can write
\begin{align*}
\P(A\cap \{Y_{t+1}<an\})&=
\E[\1_A  \1_{Y_{t+1}<an\}}]=\E[\1_A\E[\1_{Y_{t+1}<an\}}|\mathcal{F}_t]]\\
&=\E[\1_A \phi(Y_t,an)]\le \E [\1_A c/n]=c/n\P(A),
\end{align*}
so $\P(Y_{t+1}\ge an| A)\ge 1-\frac{c}n$.\\
By natural induction, it follows that for $A_t=\miniop{t}{\cap}{i=1}\{Y_{t}\ge na^t\}$, we have
$$\P^n(A_t)\ge\miniop{t-1}{\prod}{i=0}\left(1-\frac{c}{na^i}\right),$$
then
$\P^n(\tau=+\infty)\ge\P^n(\forall t\ge 0\quad Y_{t}\ge na^t)\ge\prod_{i=0}^{+\infty} \left(1-\frac{c}{na^i}\right)>0.$
\end{proof}

Some remarks:
\begin{itemize}
\item Obviously, the bound $1-\frac{c}n$ is very bad, coming from the Tchebitchev inequality. We were doing better with the Höffding inequality, but that is sufficient for our purpose.
\item  The same pattern can be applied to demonstrate that survival is possible for a multitype Galton--Watson process whose fertility matrix has a spectral radius strictly greater than~1 (see for example~\cite{Garet-livre}).
\end{itemize}
\subsection{Survival is a local property}

\begin{theo}
\label{equivalence}
Let $(Y_n)_{n\ge 0}$ be a  Galton--Watson process with offspring distribution~$\nu$.
Suppose that $\nu(0)>0$. Then there is an equivalence between:
\begin{itemize}
\item $\exists N,T\ge 1\quad \P^N(Y_T\ge 2N)>\frac12$.
\item $\P^1(\tau=+\infty)>0$.
\end{itemize}
\end{theo}

The event $\{Y_T\ge 2N\}$ only depends on what happens in a finite time box. Thus, it can be considered to be a local event, which will be useful to get some continuity with respect to the parameters of the model. \\

Before starting the proof, let us give the main ideas:
\begin{itemize}
\item For the direct implication, the idea is to compare the chain with a supercritical  Galton--Watson process, then conclude with the help of Theorem~\ref{surcritique}.
\item The reverse implication is quite simple, because one essentially has to prove that the number of particles explodes as soon as the process survives.
  However, it must be kept in mind that if the local event is more complicated, this part will actually be the most difficult one.
\end{itemize}

\begin{lemme}
 If there exist $a>0$ and $N\ge 1$ such that $a\P^N(Y_1\ge aN)>1$, then $\P^1(\tau=+\infty)>0$.
\end{lemme}
\begin{proof}
Let $X_i^n$ be i.i.d. with $\nu$ as common distribution.
Let $M_0=1$, $Y_0=N$, and then 
$$\forall n\ge 0\quad Y_{n+1}=\sum_{1\le i\le Y_n}X_i^n\text{ and }M_{n+1}=\sum_{i=1}^{M_n} aB_i^n,$$
with $B_i^n=\1_{\{X^n_{(i-1)N+1}+\dots X^n_{iN}\ge aN\}}$.\\
We prove by natural induction that $Y_n\ge NM_n$ for each $n\ge 0$. Indeed, if $Y_n\ge NM_n$, it follows that
\begin{align*}
  Y_{n+1}=\sum_{1\le i\le Y_n}X_i^n\ge \sum_{1\le i\le NM_n}X_i^n&=\sum_{i=1}^{M_n}(X^n_{(i-1)N+1}+\dots X^n_{iN})\\&\ge \sum_{i=1}^{M_n} aNB_i^n=NM_{n+1}.
  \end{align*}
We note that $(M_n)$ is a Galton--Watson process, and its fertility is given by\\ $m=\E[aB_i^n]=a\P^N(Y_1\ge aN)>1$, then it may survive by Theorem~\ref{surcritique}.
Since  $Y_n\ge NM_n$, the process $(Y_n)$ may survive too.
\end{proof}
Note that the proof of the lemma relies on a coupling argument: we make live on the same space $(Y_n)_{n\ge 0}$ and a Galton--Watson process with offspring distribution $(1-q)\delta_0+q\delta_{a}$, where $q=\P^N(Y_1\ge aN)$.\\
This step can be seen as a static renormalization: with the  help of the local events $\{X^n_{(i-1)N+1}+\dots X^n_{iN}\ge aN\}$, we build a growth process involving Bernoulli variables, in such a way that
\begin{itemize}
\item The process using Bernoulli variables is known to be able to survive;
\item  The process using Bernoulli variables is dominated by the process that we study.
  \end{itemize}

\begin{proof}[Proof of Theorem~\ref{equivalence}]
By corollary~\ref{souschaine}, $(Y_{nT})_{n\ge 0}$ is a Galton--Watson process. So we can apply the Lemma with $a=2$: $(Y_{nT})_{n\ge 0}$ may survive, thus  $(Y_{n})_{n\ge 0}$ may survive also.

Conversely, let us suppose that $\nu(0)>0$, and $\P^1(\tau<+\infty)<1$.\\
Since $\P^N(\tau<+\infty)=\P^1(\tau<+\infty)^N$, there exists $N$ with $\P^N(\tau<+\infty)<1/2$.

We have noted that $\P^N(\tau<+\infty|\mathcal{F}_t)=\P^1(\tau<+\infty)^{Y_t}$.\\
Since $\P^1(\tau<+\infty)\ge \P^1(Y_1=0)=\nu(0)>0$, we can write
$$Y_t=\frac{\log \P^N(\tau<+\infty|\mathcal{F}_t)}{\log \P^1(\tau<+\infty)}.$$
Now, the Martingale convergence Theorem ensures that
$$\E^N[\1_{\{\tau<+\infty\}}|\mathcal{F}_t]=\P^N(\tau<+\infty|\mathcal{F}_t)\to \1_{\{\tau<+\infty\}}\quad\P^N\text{ a.s.}$$ when $t$ tends to infinity. \\
Particularly, on the event $\{\tau=+\infty\}$, 
$\P^N(\tau<+\infty|\mathcal{F}_t)$ almost surely tends to  $0$ and
$Y_t$ almost surely tends to infinity. Therefore, the following inequality holds $\P^N$-almost surely:
$$\1_{\{\tau=+\infty\}}\le \miniop{}{\liminf}{t\to +\infty}\1_{\{Y_t\ge 2N\}}.$$
With the Fatou Lemma, it follows that
$$\P^N(\tau=+\infty)=\E^N(\1_{\{\tau=+\infty\}})\le \miniop{}{\liminf}{t\to +\infty}\E^N[\1_{\{Y_t\ge 2N\}}]= \miniop{}{\liminf}{t\to +\infty}\P^N(Y_t\ge 2N).$$
Since $\P^N(\tau=+\infty)>1/2$, there exists $T$ such that $\P^N(Y_T\ge 2N)>1/2$.
\end{proof}

\subsection{The critical case}

\begin{theo}
  If $\nu(0)>0$ and $m=1$, then $\P^1(\tau=+\infty)=0$.
\end{theo}

\begin{proof}[First proof]
It is sufficient to note that for every $N,T\ge 1$, we have $$\P^N(Y_T\ge 2N)\le\frac{\E^N(Y_T)}{2N}=\frac{N}{2N}=\frac12,$$
then apply the converse part in Theorem~\ref{equivalence}.
\end{proof}
We now present another line of proof, somewhat longer, but also more robust.
It was used in Garet--Marchand~\cite{GM-BRW} and Gantert--Junk~\cite{Gantert} for the study of some branching random walks.

The first proof is not robust because it exploits the fact that we exactly know how to characterize the critical parameter for survival.
However, in many growth models, the critical parameter can not be given explicitly.
The idea is then: having shown that survival is characterized by the fact that a local event has a fairly high probability, we reason by contradiction and suppose that there is survival at the critical point for a certain parameter.
Then, with a slight modification of the local event, we can, by continuity, exhibit a model of the same family that is a little weaker, for which the local event still has a probability that is large enough to ensure survival, but which must nevertheless die because its parameter has become subcritical.

\begin{proof}[Second proof]
By contradiction, let us assume that we have $\nu(0)>0$, $m=1$ and also $\P^1(\tau=+\infty)>0$.

By Theorem~\ref{equivalence} (converse implication), one can choose $n$ and $T$ such that  $\P^N(Y_T\ge 2N)>\frac12$.

The idea is to provide a coupling with a subcritical process.
Let $(X_i^n)_{i,j\ge 1}$, $(B_i^n)_{i,j\ge 1}$ be independent variables with $X_i^n\sim \nu$,  and the $(B_i^n)_{i,j\ge 1}$'s are  Bernoulli with parameter $p$. Define $Y_0=N$, $Y^p_0=N$, then 

$$\forall n\ge 0\quad Y_{n+1}=\sum_{1\le i\le Y_n}X_i^n\text{ and }Y^p_{n+1}=\sum_{1\le i\le Y^p_n}B_i^n X_i^n.$$

By monotonicity,
$$\miniop{}{\lim}{M\to +\infty}\P^N( \max(Y_i,0\le i\le T)\le M,Y_T\ge 2N)=\P^N( Y_T\ge 2N)>1/2,$$
so there exists $M$ such that $\P( \max(Y_i,0\le i\le T)\le M,Y_T\ge 2N)>1/2$.
We have then
\begin{align*}
\P(Y^p_T\ge 2N)&\ge \P(Y_T\ge 2N,\forall i\le T\quad Y^p_i=Y_i)\\
& \ge \P\left( \begin{array}{l}\max(Y_i,0\le i\le T)\le M,Y_T\ge 2N,\\\forall (t,i)\in\{0,\dots, T-1\}\times\{1,\dots,M\} \quad B_i^t=1\end{array}\right)\\&=\P( \max(Y_i,0\le i\le T)\le M,Y_T\ge 2N)p^{TM}
\end{align*}
Taking $p<1$ large enough, we have $$\P( \max(Y_i,0\le i\le T)\le M,Y_T\ge 2N)p^{TM}>1/2,$$ so $\P(Y^p_T\ge 2N)>1/2$.
But $(Y^p_t)$ is a Galton--Watson process with offspring distribution $B_1^1X_1^1$ and initial distribution $\delta_N$, so by Theorem~\ref{equivalence} (direct implication), this Galton--Watson process may survive.
However $$\E[B_1^1X_1^1]=\E[B_1^1]\E[X_1^1]=pm=p<1,$$ so by Theorem~\ref{mort}, the process can not survive. This is a contradiction.

\end{proof}

The Galton-Watson process has the particularity that the survival domain can be described explicitly.
This is obviously very practical, but it may cast doubt on the generality of the proof technique we are presenting.

In fact, this technique is more often applied to models where the critical value is unknown, the most emblematic being the contact process or directed percolation.
But in these models, proving the existence of a high probability for the local event in question is often quite technical, requiring many steps.

We therefore present a simpler model, which allows us to demonstrate the power of the method in a model that is not completely solvable, and which we believe is sufficiently rich to be of interest.

\section{Application to a cooperative model}

We describe a cooperative model with two species by a Markov chain $((X_n,Y_n)_{n\ge 0})$ with values in $\N^2$, given by the conditional laws

\begin{align*}
\mathcal{L}((X_{n+1},Y_{n+1})|\mathcal{F}_n)&=\mu_{(X_n,Y_n)}\text{ with }\mu_{(x,y)}=\mu_{1,1}^{*(x+y)}*\mu_{2,1}^{\min(x,y)} \otimes \mu_{2,1}^ {*(x+y)}*\mu_{1,2}^{\min(x,y)},
\end{align*}
where $\mathcal{F}_n=\sigma((X_k,Y_k)_{0\le k\le n})$. In other words, we have the representations

\begin {align*}X_{n+1}=\sum_{k=1}^{X_n+Y_n} W^{(1,1)}_{k+1}+\sum_{k=1}^{\min(X_n,Y_n)} W^{(2,1)}_{k+1}\\
Y_{n+1}=\sum_{k=1}^{X_n+Y_n} W^{(1,2)}_ {k+1}+\sum_{k=1}^{\min(X_n,Y_n)} W^{(2,2)}_{k+1},\\
\end{align*}
where the $W^{(i,j)}_{n}$ are independent variables, $\mu_{i,j}$ being the distribution of $W^{(i,j)_n}$.

This is a model with two types of individuals, each of which can reproduce asexually, with offspring potentially of both types, plus a sexual component involving faithful pairing of individuals of both types.

To simplify matters, we assumed that the distribution of an individual's offspring through asexual reproduction did not depend on its type.

Let $E$ be the set of quadruplets
$\ mu=(\mu_{i,j})_{1\le i,j\le 2}\in \mathcal{P}$, where each $\mu_{i,j}$ is a probability measure on $\N$.

In what follows, the letter $\mu$ denotes a quadruplet of $E$ and we also denote by $\P^{(x,y)} _{\mu}$ the law of the process starting from state $(x,y)$ and subject to the dynamics based on $\mu=(\mu_{i,j})_{1\le i,j\le 2}\in E$. 

Let us recall some classic definitions.
 
If $X$ is a finite or countable set, a function $f$ from $\R^X$ to $\R$ is said to be increasing in the product order if
\[(\forall k\in X;\quad x_k\le y_k) \Longrightarrow f(x)\le f(y).\]
If $\alpha$ and $\beta$ are probability measures on $\R^X$, we say that $\beta$ stochastically dominates $\alpha$ and we write
$\alpha\lestoch\beta$ if for every function $f : \mathbb{R}^X \to \mathbb{R}$ increasing in the product sense, we have
\[
\int_{\mathbb {R}^X} f \, \mathrm{d}\alpha \;\le\; 
\int_{\mathbb{R}^X} f \, \mathrm{d}\beta.
\]

We can then note that the process is superadditive.
Specifically, if $x,x',y,y'$ are any natural numbers, we have the inequality on the transition laws
\begin{align*}
\mu_{(x,y)}*\mu_{(x',y')}& =\mu_{1,1}^{*(x+y)}*\mu_{2,1}^{\min(x,y)}*
\mu_{1,1}^{*(x'+y')}*\mu_{2,1}^{\min(x',y')} \\
&\otimes \mu_{2,1}^{*(x+y)}*\mu_{1,2}^{\min(x,y)}* \mu_{2,1}^{*(x'+y')}*\mu_{1,2}^ {\min(x',y')}.\\
& =\mu_{1,1}^{*(x+y+x'+y')}*\mu_{2,1}^{\min(x,y)+\min(x',y')}\\
&\otimes \mu_{2,1}^{*(x+y+x'+y')}*\mu_{1,2}^{\min(x,y)+\min(x',y')}.\\
&\lestoch \mu_{1,1}^{*(x+y+x'+y')}*\mu_{2,1}^{\min(x+x',y+y')}\\
&\otimes \mu_{2,1}^{*(x+y+x'+y')}*\mu_{1,2}^{\min(x+x',y+y')}=\mu_{x+x',y+y'}.\\
\end{align*}

This inequality can be classically transferred to process laws: whatever
the measures $\mu=(\mu_{i,j})_{1\le i,j\le 2}$ and integers $x,x',y,y'$, we have 
\begin {align}
\label{inegloi}
\P_{\mu}^{(x,y)}*\P_{\mu}^{(x',y')}\lestoch \P_{\mu}^{(x+y',y+y')}.
\end{align}

In particular, if $x\ge x'$ and $y\ge y'$, then $\P_\mu^{(x,y)}\gestoch \P_\mu^{(x',y')}$.

The model we are studying is a special case of the Galton-Watson multitype bisexual branching process, as defined by Fritsch, Villemonais, and Zalduendo~\cite{zbMATH07967668}.~\footnote{Following their nomenclature, we work with the mating function $\xi(x,y)=(x+y,\min(x,y))$.}
However, the families we propose to study in more detail do not fall within the scope of their main results, which require strong irreducibility assumptions.

\subsection{Continuity results}

In what follows, we set \[S_n=X_n+Y_n,\quad Z_n=\inf(X_n,Y_n)\text{ and }\tau=\inf\{n\ge 0;Z_n=0\}\].

We begin by deducing from~\eqref{inegloi} a lemma that will be very useful later on:
\begin{lemme}\label{grandpas}
Let $p\in [0,1]$, $x,y,k$ be natural numbers, and $N$ and $T$ be nonzero natural numbers.
Then
\begin{align*} \P_{\mu}^{(x,y)}(Z_T\ge kN)&\ge \gamma^{*a}([k,+\infty[),
\end{align*}
where $a=\min(\floor{\frac{x}N},\floor{y}N)$ and $\gamma$
is the distribution of $\floor{\frac{Z_{T}}{N}}$ under $\P_{\mu}^{(N,N)}$.
\end{lemme}
\begin {proof}
Let $(\tilde{X}_1,\tilde{Y_1}),\dots, (\tilde{X}_a,\tilde{Y_a})$ be $a$ independent random vectors following the distribution of $(X_T,Y_T)$ under $\P_{\mu}^{(N,N)}$.
Setting $\tilde{Z}_k=\min(\tilde{X} _k,\tilde{Y}_k)$, we have
\begin{align*}
\P_{\mu}^{(x,y)}(Z_T\ge kN)&\ge \P(\tilde{X}_1+\dots \tilde{X}_a\ge kN,\tilde{Y}_1+\dots \tilde{Y}_a\ge kN)\\
&\ge \P(\tilde{Z}_1+\dots \tilde {Z}_a\ge kN)\\
&\ge \P( \floor{\frac{\tilde{Z}_1}N}+\dots \floor{\frac{\tilde{Z}_a}N}\ge k) =\gamma^ {*a}([k,+\infty)).
\end{align*}
Thus
\begin{align*}
\P_{\mu}^{(N,N)}(\floor{\frac{Z_{(n+1) T}}{N}}\ge k|\mathcal{F}_{nT})&= \P_{\mu}^{(X_{nT},Y_{nT})}(\floor{\frac{Z_{T}}{N}}\ge k|\mathcal{F}_{nT})\\ &\ge \gamma^{* \floor{\frac{Z_{nT}}{N}} }([k,+\infty))
\end{align*} 

\end{proof}

We can now state the locality lemma, analogous to Theorem~\ref {equivalence}.
Let's start with a definition.
\begin{definition}
Let $\mu\in E$. We say that $\mu$ has the locality property if we have 
equivalence between
\label{leqvencore}
\begin{itemize}
\item $\P_{\mu}^{(1,1)}(\tau=+\infty)>0$.
\item $\exists N\ge 1,T\ge 1\quad \E_{\mu}^{(N,N)}(\floor{\frac{Z_T}N})>1$.
\end{itemize}
We also say that a subset $F$ of $E$ has the locality property if all its elements have the locality property. 
\end{definition}

\begin{lemme}
\label{lemmebienabstrait}
Let $X$ be a sequential topological space, $(\Omega,{\mathcal{F}})$, a measure space.
Let $\mathcal{M}(\Omega,\mathcal{F})$ be the set of probability measures on $(\Omega,{\mathcal{F}})$. We equip $\mathcal{M}(\Omega,\mathcal{F})$ with the total variation distance:
\[d_{VT}(\mu,\nu)=2\sup_{A\in\mathcal{F}}|\mu(A)-\nu(A)|.\]
Let $D$ be a countable set.
We assume that for all $k\in D$, the mapping 
\begin{align*}
X&\to \mathcal{M}(\Omega,\mathcal{F})\\
x &\mapsto \nu^k_x
\end{align*}
is continuous (for the topology of the total variation distance).
We set

\[(\nu!\mu)_x=\sum_{k\in D} \mu(k)\nu^k_x ,\]

Then, the mapping
\begin{align*}
X\times \mathcal{M}(D,\mathcal{P}(D)) &\to \mathcal{M} (\Omega,\mathcal{F})\\
(x,\mu)&\mapsto (\nu!\mu)_x
\end{align*}
is continuous.
\end{lemme}
\begin{proof}
Suppose that $\mu_n$ converges to $\mu$ and that $(x_n)$ converges to $x$.
Let $A\in\mathcal{F}$.
We have
$$\begin {aligned}
(\nu!\mu)_x(A)-(\nu!\mu_n)_{x_n}(A)&=\sum_{k\in D} \mu(k)\left(\nu_x^k(A)-\nu_{x_n}^k(A)\right)\\
&\quad + \sum_{k\in D}\left(\mu(k)-\mu_n(k)\right) \nu_{x_n}^k(A)
\end{aligned}$$
from which
$$\begin{aligned}
|(\nu!\mu)_x(A)-(\nu!\mu_n)_{x_n}(A)|&\le \frac12\sum_{k\in D} \mu 
(k) d_{VT}(\nu_x^k,\nu_{x_n}^k)\\
&\quad + \sum_{k\in D}\left|\mu(k)-\mu_n(k)\right|
\end {aligned},$$
then

\[d_{VT}((\nu!\mu)_x,(\nu!\mu_n)_{x_n})\le \sum_{k\in D} \mu(k) d_{VT}(\nu_x^k,\nu_{x_n}^k)+2d_{VT}(\mu,\mu_n)\]
The first term tends towards $0$ by dominated convergence, the second thanks to Scheffé's lemma.
Thus, $(\nu!\mu_n)_{x_n}$ converges towards $(\nu!\mu)_{x}$. As the space is sequential, this shows that the application $(x,\mu)\mapsto (\nu!\mu)_x$ is continuous.
\end{proof}

We deduce the continuity theorem:
\begin{theo}
\label{theocontinuite}
Let $F$ be a subset of $E$ that has 
the locality property. Then
$$S=\{\mu\in F; \P^{(1,1)}_{\mu}(\tau=+\infty)>0\}.$$
is an open set in $F$.
\end{theo}

\begin{proof}
According to the definition~\ref{leqvencore} of locality, we have 
\begin{align}
\label{egalitecont}
S=\miniop{}{\cup}{N\ge 1,T\ge 1} \{\mu\in F; \E_{\mu}^{(N,N)}(\floor{\frac{Z_T}N})>1\}.
\end{align}

To conclude, it suffices to prove that for $N,T\ge 1$, $\mu\mapsto \E_{\mu}^{(N,N)}(\floor{\frac{Z_T}N})$ is a semi-continuous function from below.

We begin by showing that for all $n$, the distribution of $(X_n,Y_n)$ depends continuously on $\mu$.
The set $E$ where $\mu$ lives is metrizable; it is therefore sequential and we can apply Lemma~\ref{lemmebienabstrait}.

For $\mu\in E$, we denote by $\mu_i$ the reproduction law of a unit of type $i$ (this is the distribution of $(W^{(1,i)},W^{(2,i)})$).

We then set, for $(i,j)\in\N^2$: $\displaystyle \nu_{\mu}^{i,j}=\mu_1^{*i}*\mu_2^{*j}$.

It is easy to see that whatever $i$ and $j$ may be, $\mu\mapsto \nu_{\mu}^{i,j}$ is continuous.

Let $L^\mu_n$ be the distribution of $(X_n,Y_n)$ under $\P^{(N,N)}_{\mu}$.

We have the recurrence formula:

\[L^{\mu}_ {n+1} =\nu_{\mu}!(L^{\mu}_n)_{\xi},\]

where $(L^{\mu}_n)_ {\xi}$ denotes the image distribution of $(L^{\mu}_n)$ by the pairing function $\xi(x,y)=(x+y,\min(x,y))$. With Lemma~\ref{lemmebienabstrait}, this allows us to establish by recurrence that
$L^\mu_n$ depends continuously on $\mu$.

If we set \(\displaystyle F_{N,T,k}(\mu)=\sum_{1\le i,j\le k} L^{\mu}_T(\{(i,j)\}) \floor{\frac{\min(i,j)}{N}},\)

it is clear that $\mu\mapsto F_{N,T,k}(\mu)$ is continuous, from which we deduce that $S$ is an open set in $F$, since
\[S=\miniop{}{\cup}{N\ge 1,T\ge 1,k\ge 1} \{\mu\in F; F_{N,T}(\mu)>1\}.\]
\end{proof}

\subsection{Specific models}

We now work with the additional hypothesis
\begin{align}
\mu_{2,1}=\delta_0
\end{align}
which means that the union of two individuals of different types can never give rise to an element of type 1. The recurrence thus takes the form

\[\begin{aligned}X_{n+1}&=\sum_{k=1}^{X_n+Y_n} W^{(1,1)}_k\\
Y_{n+1}&=\sum_{k=1}^ {X_n+Y_n} W^{(1,2)}_k+\sum_{k=1}^{\min(X_n,Y_n)} W^{(2,2)}_k\\
\end{aligned}\]

Since the degree of generality is still too high for a detailed analysis, we focus on two specific families:

\begin{itemize}
\item Family A: $\mu_{1,2}=\mu_{2,1}=\delta_0$; $\mu_{2,2}(0)\ne 0$;
\item Family B: $\mu_{1,2},\mu_{2,2}\ne \delta_0$, $\mu_{2,1}=\delta_0$, $\mu_{1,1}(0)>0$;$\mu_{1,2}(0)>0$.
\end{itemize} 
\vspace{0.3cm}

In family A, the first component of the pair is the number of type 1 elements, which are produced asexually by representatives of both types; while the second component, type 2 elements, are produced by an encounter between type 1 elements and type 2 elements.

Note that if at a given moment there are no more type 1 particles or no more type $2$ particles, the type $2$ particles disappear without any possibility of reappearing.

On the other hand, if type $2$ particles disappear and type $1$ particles remain, the process of type $1$ particles then behaves like a Galton-Watson process with reproduction law $\mu_{2,2}$: their survival is possible if and only if the average number of descendants $\int x\ d\mu_{2,2} $ exceeds 1 (we have excluded the case where $\mu_{2,2}=\delta_0$).

In family B, asexual reproduction gives rise to both types, so both types are guaranteed to survive simultaneously. However, we expect the second type to be observed more frequently than the first.

\subsection{Locality of models}

\begin{lemme}
\label{lalocalite}
The elements $\mu$ of family $A$ that satisfy
$\mu_{1,1}(0)>0$ and the elements of family $B$ have the locality property.

\end{lemme}

\begin{proof}
Suppose $\P_{\mu}^{(1,1)}(\tau=+\infty)>0$.
From~\eqref{inegloi}, we deduce
that if $\min(x,y)\ge N$, then
\begin{align}
\label{sibeaucoup}
\P^{(x,y)}(\tau<+\infty)&\le \P^{(1,1)}(\tau<+\infty)^N
\end{align}

Then, thanks to~\eqref{sibeaucoup}, we can find $N$ such that $\P_{\mu}^{(N,N)}(\tau=+\infty)>\frac34$.

\begin{itemize}
\item Case A:
We have
\begin{align*}
\P_{\mu}^{Z(N,N)}(\tau<+\infty|\mathcal{F}_n)&\ge \P_{\mu}^{(N,N)}(Y_{n+1}=0|\mathcal{F}_n)=\begin{cases}1 &\text{ if }Z_n=0\\ (\mu_{2,2}(0))^{Z_n}& \text{ otherwise.}\end{cases}
\end{align*}
On the event $\{\tau=\infty\}$, $\P_{\mu}^{(N,N)}(\tau<+\infty|\mathcal{F}_n)$ converges almost surely to $0$, so $Z_n$ tends $\P_{\mu}^{(N,N)}$ almost surely to infinity. Thus, $\1_{\{\tau=+\infty,Z_n<2N\}}$ tends almost surely to $0$, by dominated convergence, $\P_{\mu}^{(N,N)}(\tau=+\infty,Z_n<2N)$ tends to $0$, which means
that $\P_{\mu}^{(N,N)}(Z_T\ge 2N)>\frac12$ for sufficiently large $T$.
We can deduce that 
\begin{align*}
\E^{(N,N)}_{\mu}(\floor{\frac{Z_{T}}{N}})&\ge \E^{(N,N)}_{\mu}(2\1_{\{Z_T\ge 2N\}})=2\P^{(N,N)}_{\mu}(Z_T\ge 2N)>1.
\end{align*}
\item Case B:
We have
\begin{align*}
\P_{\mu}^{(N,N)}(\tau<+\infty|\mathcal{F}_n)&\ge \P_{\mu}^{(N,N)}(Y_{n+1}=0|\mathcal{F}_n)=\mu_{1,2}(0)^{X_n+Y_n}\mu_{2,2} 
(0)^{\min(X_n,Y_n)}\\&\ge (\mu_{1,2}(0)\mu_{2,2}(0))^{X_n+Y_n}
\end{align*}
On the event $\{\tau=+\infty\}$, $\P_{\mu}^{(N,N)}(\tau<+\ infty|\mathcal{F}_n)$ converges almost surely to $0$, so $S_n=X_n+Y_n$ tends $\P_{\mu}^{(N,N)}$ almost surely to infinity. Thus, if $M$ is any natural number, $\1_{\{\tau=+\infty,S_n<M\}}$ tends almost surely to $0$, by dominated convergence, $\P_{\mu}^{(N,N)}(\tau=\infty,S_n<M)$ tends towards $0$, which means that we can choose $T$
such that $\P_{\mu}^{(N,N)}(S_{T-1}\ge M)>\frac34$ for $T$ sufficiently large.
Suppose that $M$ is chosen such that for $i\in\{1,2\}$, $\mu_{1,i}^{*M}([0,2N]) <1/8$.
We deduce that
 
\begin{align*}
\E^{(N,N)}_{\mu}(\floor{\frac{Z_{T}}{N}})&\ge \E^{(N,N)}_{\mu}(2\1_{\{Z_T\ge 2N\}})=2\P^{(N,N)}_{\mu}(Z_T\ge 2N).
\end{align*}
However,
\begin{align*}
\P^{(N,N)}_{\mu}(Z_T\ge 2N|\mathcal{F}_{T-1})&\ge 1-\sum_{i=1}^2\P_\mu\left(\sum_{k=1}^{S_{T-1}} W_{T}^{1,i}\le 2N\Big|\mathcal{F}_{T-1}\right)\\
&\ge 1-\sum_{i=1}^2 \mu_{1,i}^{*S_{T-1}}([0,2N])\\
&\ge \frac34\1_{\{S_{T-1}\ge M\}}
\end{align*}

\end{itemize}

For the converse, there is no need to treat the two models separately.
Now suppose that there exist integers $N$ and $T$ such that
$\E^{(N,N)}_{\mu}(\floor{\frac{Z_{T}}{N}})>1$.

We will show that the process $(\floor{\frac{Z_{nT}}{N}})_{n\ge 0}$ stochastically dominates a supercritical Galton-Watson process.

Let $\gamma$ be the distribution of $\floor{\frac{Z_{T}}{N}}$ under $\P_{\mu}^{(N,N)}$.

Using Lemma~\ref{grandpas} and the Markov property, we have for all $k\ge 0$ and all natural numbers $n$:

\begin{align*}
\P_{\mu}^{(N,N)}(\floor{\frac{Z_{(n+1)T}}{N}}\ge k|\mathcal{F}_{nT})&= \P_{\mu}^{(X_{nT},Y_{nT})}(\floor{\frac{Z_{T}}{N}}\ge k|\mathcal{F}_{nT})\\ &\ ge \gamma^{* \floor{\frac{Z_{nT}}{N}} }([k,+\infty))
\end{align*}

This shows that $(\floor{\frac{Z_{nT}}{N}})_{n\ge 0}$ stochastically dominates a Galton-Watson process with reproduction law $\gamma$, which is supercritical according to the condition on the expectation.
This implies that the process survives with strictly positive probability.
\end{proof}

\subsubsection{Parametric study of family A}

We set $Q=\E_\mu(W^{1,1}_1)=\int_{\R} x\ d\mu_{1,1}(x)$ and
$P=\E_\mu(W^{2,2}_1)=\int_{\R} x\ d\mu_{2,2}(x)$

\begin{lemme}
\label{paspossible}
In family A, if
$Q(1+P)<1$ or $P<1$, survival is impossible.
\end{lemme}
\begin{proof}
First, we have
\begin{align*}
\E_{\mu}[Z_{n+1}|\mathcal{F}_n]&\le \E_{\mu}[Y_{n+1}|\mathcal{F}_n]=P Z_n,
\end{align*}
so $\E_{\mu}[Z_n]\le P^n \E_{\mu}(Z_0)$ and with Borel-Cantelli's lemma,
$Z_n$ tends almost surely to $0$ for $P<1$.

We also have
\[\begin{aligned}
\E_\mu[X_{n+1}|\mathcal{F}_n]&=(X_n+Y_n)\E_\mu(W_1^{1,1})=Q(X_n+Y_n)\\
\E_\mu[Y_{n+1}|\mathcal{F}_n]&=\min(X_n,Y_n)\E_\mu(W_1^{2,2})=P\min(X_n,Y_n)\le P X_n
\end{aligned}\]
Reintegrating, we obtain that
\[ \begin{pmatrix}
\E_\mu [X_{n+1}]\\[0.3em]
\E_\mu[Y_{n+1}]
\end{pmatrix}
\le
\begin{pmatrix}
Q & Q \\[0.3em]
P & 0
\end{pmatrix}
\begin{pmatrix}
\E_\mu[X_{n}]\\[0.3em]
\E_\mu[Y_{n}]
\end{pmatrix}\]
Thus,
\[ \begin{pmatrix}
\E_\mu[X_{n}]\\[0.3em]
\E_\mu[Y_{n}]
\end{pmatrix}
\le
\begin{pmatrix}
Q & Q \\[0.3em]
P & 0
\end{pmatrix}^n
\begin{pmatrix}
\E_\mu[X_{0}]\\[0.3em]
\E_\mu[Y_{0}]
\end{pmatrix}\]
As before, Borel-Cantelli's lemma shows that survival is impossible if the spectral radius of the matrix
$M=\begin{pmatrix}
Q & Q \\[0.3em]
P & 0
\end{pmatrix}$ is less than~1.
The characteristic polynomial of the matrix is
$$\chi_M(X)=X^2-QX-PQ.$$
The discriminant $Q^2+4PQ$ is strictly positive, the two eigenvalues $r_1$ and $r_2$ are real; the one with the larger modulus (which we denote $r_1$) is positive; the other is negative. Let $s_i=r_i-1$.
We have $s_2<r_2<0$, so
\[ (r_1<1)\iff (s_1<0) \iff s_1s_2>0.\]
Now $s_1s_2=(1-\lambda_1)(1-\lambda_2)=\chi_M(1)=1-Q-PQ$, so
survival is impossible if $Q (1+P)<1$.
\end{proof}

\begin{lemme}
In family A, if $Q(1+P)>1$ and $P>1$, survival is possible.
\end{lemme}

\begin{proof}
Suppose $Q(1+P)>1$ and $P>1$ and show that survival is possible.

First, note that if survival is possible for a pair of reproduction laws, it will also be possible for a pair of laws that is stochastically larger. 
This allows us to reduce the case to $Q<\frac{P}2$. Indeed, if we do not have $Q<\frac{P}2$, since $\frac1{1+P}<\frac12<\frac{P}2$, we can find $Q'$ such that $\frac1{1+P}<Q'<\frac{P}2$. We then have $Q'<\frac{P}2\le Q$, $Q'(1+P)>1$.
We can then replace $\mu_{1,1}$ by $\mu'_{1,1}=\frac{Q-Q' }{Q}\delta_0+\frac{Q'}{Q}\mu_{1,1}$: this distribution satisfies the conditions imposed on the expectation and is stochastically dominated by $\mu_{1,1}$.
Similarly, we can reduce this to the case where the distributions of $\mu_{1,1}$ and $\mu_{2,2}$ have finite support. Indeed, if $W^{i,i} $ follows the distribution $\mu_{i,i}$ and we set
$W^{i,i}_n=W^{i,i}\wedge n$, $Q_n=\E(W^{1,1}_n)$ and $P_n=\E(W^{2,2}_n)$,
then for sufficiently large $n$, we have $Q_n(1+P_n)>1$, $P_n>1$ and $Q_n<\frac {P_n}2$.

From now on, we assume that $Q(1+P)>1$, $P>1$, $Q<P/2$, and that the reproduction laws have finite support.

The dynamics of $(X_n,Y_n)_{n\ge 0}$ is given by the recurrence

\begin{align*}X_ {n+1}=\sum_{k=1}^{X_n+Y_n} W^{(1,1)}_k &\quad
Y_{n+1}=\sum_{k=1}^{\min(X_n,Y_n)} W^{(2,2)}_k
\end{align*}
Let $X'_0=X_0$, $Y'_0=Y_0$, then
\begin{align*}X'_{n+1}=\sum_{k=1}^{X'_n+Y'_n} W^{(1,1)}_k &\quad
Y'_{n+1}=\sum_{k=1}^{X'_n} W^{(2,2)}_k,
\end{align*}

as well as \[S=\{\forall n\ge 0, \min(X_n,Y_n)>0)\}, \quad S'=\{\forall n\ge 0, \min(X'_n,Y'_n)>0)\},\]
and also \(M_n=\{\forall k\ge n;\quad X'_k\le Y'_k\}\).
It is easy to see that we have the inclusion 

\[M_0 \subset \{\forall n\ge 0,\quad (X_n,Y_n)=(X'_n,Y'_n)\}.\]

We deduce that \(\P(S)\ge \P(S',M_0)\).

Now, $(X'_n,Y'_n)_{n\ge 0}$ is a two-type Galton-Watson chain whose reproduction matrix is precisely the transpose of the matrix $M$.
appearing in Lemma~\ref{paspossible}.

Harris's theory tells us that the survival of the chain $(X'_n,Y'_n)_{n\ge 0}$ is possible as soon as $\rho(M)>1$. Thus, returning to the calculations made previously, we see that
$Q(1+P)>1$ implies $\lambda_1=\rho(M)>1$, and therefore $\P(S')>0$.

Still according to Harris's theory (see, for example, Harris~\cite{harris1963theory}, th. 9.2 p 44), there exists a random variable $W$ such that on the event $S'$,

\[ \left(\frac{X'_n}{\lambda^n},\frac{Y'_n}{\lambda^n}\right)\to Wv,\]

where $v$ is an eigenvector on the right for $M$ associated with the eigenvalue $\rho(M)$. We can take $v=(\lambda_1\quad P)$. Let's compare $\lambda_1$ and $P$.
As before, we form the calculation

\[(P-\lambda_1)(P-\lambda_2)=\chi_M(P)=P^2-2PQ=P(P-2Q)>0.\]

Since $P-\lambda_2>P>0$, we deduce that $P-\lambda_1>0$, or $P>\lambda_1$.

Thus, on the event $S$, $\frac{X'_n}{Y'_n}\to \frac{\lambda_1}{P}<1$, which implies that
$\P_\mu(S')=\P_\mu(S',\cup_{n\ge 1} M_n)$.
According to the sequential increasing continuity theorem, there exists $n$ such that $P_\mu(S',M_n)>0$. With the Markov property, there exist $(a,b)$
such that $\P^{(a,b)}_{\mu} (S',M_0)>0$, which gives the desired result.

\end{proof}

\begin{theo}
\label{theoA}
In model $A$, survival is possible if and only if we simultaneously have $P>1$ and $Q(1+P)>1$.
\end{theo}
\begin{proof}
We have shown that survival is impossible if $P<1$ or $Q(1+P)>1$, while survival is possible if $P>1$ and $Q(1+P)>1$.
 
The continuity theorem~\ref{theocontinuite} then allows us to say that survival is impossible on the critical line.
\end{proof}

\paragraph {\textbf{Illustration}}

To conclude the study of model A, we illustrate it with a concrete family:
we will take
\begin{itemize}
\item $\mu_{1,1}=\ber(2,q)$, so $Q=2q$;
\item $\mu_{2,2}=\ber(2,p)$, so $P=2p$.
\end{itemize}
The survival condition therefore becomes $p>\max(\frac12,\frac1{2(1+2p)})$.

If we take $N=1$ and $T=1$, the locality condition tells us that for survival, it suffices to have $\E^{(1,1)}_{\mu}(Z_1)>1$.

We then have

\begin{align*}
\E_{p,q}^{(1,1)}(Z_1)&=\int_{\R^2}\min(s,t) \ d\mu_{(1,1)}(s,t)
\end{align*}
In other words, $h(p,q)= \E_{p,q}^{(1,1)}(Z_1)$ is $\E (\min(U,V))$ where
$U$ and $V$ are independent random variables, following respectively
$\ber(4,q)$ and $\ber(p, 2)$. A simple calculation gives

$$h(p,q)=4p^2q^4-12p^2q^3+12p^2q^2-4p^2q-2pq^4+8pq^3-12pq^2+8pq.$$

We can plot the theoretical survival area, the Monte Carlo estimate of the survival property , and the area where $h(p,q)>1$ on the same graph.

\begin{figure}[h!t]
 
\centering
\begin{tabular}{cc}
\includegraphics[scale=0.42]{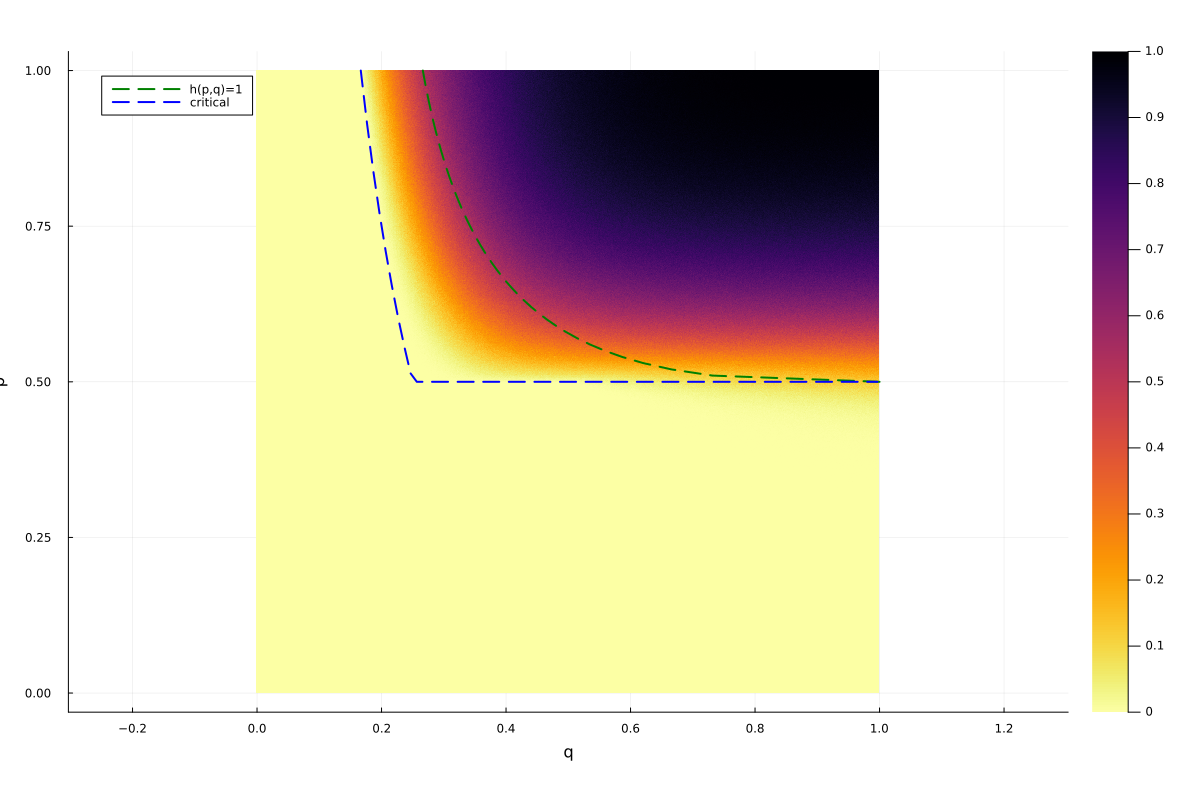}
\end{tabular}
\caption{Estimation of $\P^{(1,1)}_{p,q}(\tau>\inf\{n\ge 0;\max(X_n,Y_n)>10^8\})$. }
\end{figure}

\subsubsection{Parametric study of family B}

We  restrict ourselves to the case where $\mu_{1,1}=\mu_{1,2}$: the offspring of both types generated by asexual reproduction are then the same.

The new hypotheses are then the $B’$ hypotheses: $\mu_{1,1} =\mu_{1,2}$, $\mu_{1,1}(0)\in]0,1[]$.

We still set $Q=\int x \ d\mu_{1,1}(x)$ and $P=\int x \ d\mu_{2,2}(x)$.

\begin{theo}
Under the hypotheses $B'$, survival is possible if and only if \[Q(2+P)>1.\]
\end{theo}

\begin{proof}
The line of proof is the same as in family $A$.
We can couple the chain under study with a two-type Galton-Watson chain that dominates it and coincides with it with strictly positive probability

\[\begin{aligned}X'_{n+1}&=\sum_{k=1}^{X'_n+Y'_n} W^{(1,1)}_k\\
Y'_{n+1}&=\sum_{k=1}^{X'_n+Y'_n} W^{(1,2)}_k+\sum_{k=1}^{X'_n} W^{(2,2)}_k\\
\end{aligned}\]

The reproduction matrix associated with $(X'_n,Y'_n)$ is $N=\begin{pmatrix}
Q & Q+P \\[0.3em]
Q & Q
\end{pmatrix} $.

We have $\chi_N(X)=X^2-2QX-PQ$. $\chi_N$ has two real roots $\lambda_1$ and $\lambda_2$, the one with the larger modulus $(\lambda_1)$ is positive, the other is negative.
Since $1-\lambda_2>1$, $1-\lambda_1$ has the sign of
\[(1-\lambda_1)(1-\lambda_2)=\chi_N(1)=1-2Q-PQ,\]
so that the system $(X'_n,Y'_n)$, and, a fortiori, the system $(X_n,Y_n)$, dies out almost surely if $Q(2+P)<1$.

Conversely, if $Q(2+P)>1$, the chain $(X'_n,Y' _n)$ survives with strictly positive probability according to the Harris theory. $(Q\quad \lambda_1-Q)$ is a left eigenvector for $N$; it gives the asymptotic direction of $(X'_n,Y'_n)$ when there is survival.

Since $2Q-\lambda_2>2Q$, $2Q-\lambda_1$ has the sign of
\[(2Q-\lambda_1)(2Q-\lambda_2)=\chi_N(2Q)=-PQ<0,\] so $Q<\lambda_1-Q$: as before, on the event of survival of $(X'_n,Y'_n)$ we have $\frac{X'_n}{Y'_n}\to \frac{Q}{\lambda_1}-Q< 1$, from from which we deduce that with strictly positive probability, we have $X_n<Y_n$ for all $n$, and then on this event, $(X_n,Y_n)$ and $(X'_n,Y'_n)$ coincide and both survive.
The proof ends as in Theorem~\ref{theoA}.

\end{proof}

\paragraph{Illustration}

As for model A, we illustrate model B' in the concrete family:
\begin{itemize}
\item $\mu_{1,1}=\ber(2,q)$, so $Q=2q$;
\item $\mu_{2,2}=\ber(2,p)$, so $P=2p$.
\end{itemize}
The survival condition therefore becomes $4q(1+p)>1$.

As before, we represent the theoretical survival area with the Monte Carlo estimate of the survival property.

\begin{figure}[h!t]
\centering
\begin{tabular}{cc}
\includegraphics[scale=0.42]{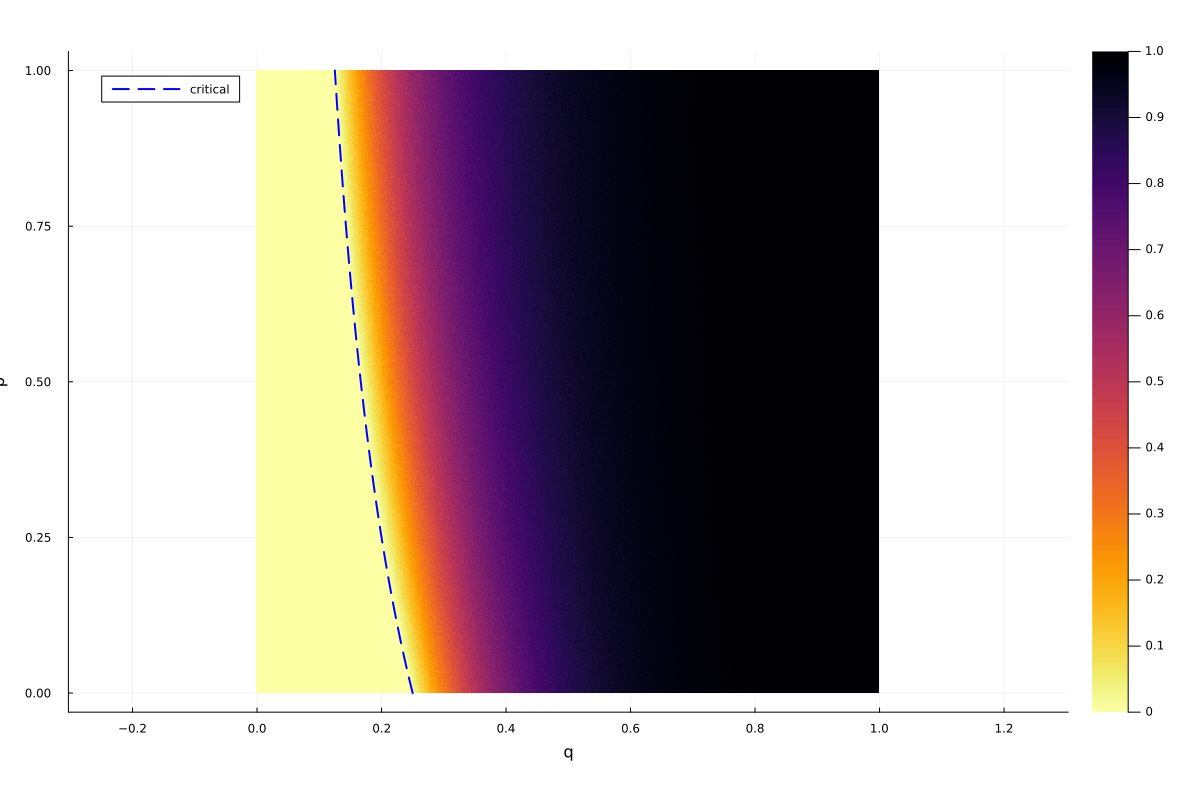}
\end{tabular} 
\caption{Estimation of $\P^{(1,1)}_{p,q}(\tau>\inf\{n\ge 0;\max(X_n,Y_n)>10^8\})$.}
\end{figure}

\newpage
\section*{Appendix: source code in Julia}

\codeJulia{code-julia-3}

\def\refname{References}
\bibliographystyle{plain}
\bibliography{galton-revisite-2025_12}

\end{document}

%% file: lst-environnements.tex
\usepackage{listings}
\usepackage{xcolor}

\lstdefinelanguage{Julia}%
  {morekeywords={abstract,break,case,catch,const,continue,do,else,elseif,%
      begin,end,export,false,for,function,immutable,import,importall,if,in,%
      macro,module,otherwise,quote,return,switch,true,try,type,typealias,%
      using,while},%
   sensitive=true,%
   alsoother={$},%
   morecomment=[l]\#,%
   morecomment=[n]{\#=}{=\#},%
   morestring=[s]{"}{"},%
   morestring=[m]{'}{'},%
}[keywords,comments,strings]%

\newtcblisting{Julia}{enhanced,breakable,listing only,listing options={language=Julia},colback=white}
\newtcbinputlisting{\codeJulia}[1]{enhanced,breakable,listing only,listing options={language=Julia},listing file={#1},colback=white}
\newtcbinputlisting{\extraitcodeJulia}[3]{enhanced,breakable,listing only,listing options={language=Julia,firstline={#2},lastline={#3}},listing file={#1},colback=white}

\lstnewenvironment{Scilab}
  {\lstset{language=Scilab,
}}
  {}
\lstnewenvironment{Bash}
  {\lstset{language=bash,
}}
  {}